\documentclass[12pt]{amsart}
\usepackage{graphicx}
\usepackage{amsmath}
\usepackage{amscd}
\usepackage{graphics}
\usepackage{latexsym}
\theoremstyle{definition}
\newtheorem{thm}{Theorem}[section]
\newtheorem{cor}[thm]{Corollary}

\newtheorem{prop}[thm]{Proposition}

\newtheorem{theo}[thm]{Theorem}
\newtheorem{lemm}[thm]{Lemma}
\newtheorem{defn}[thm]{Definition}
\newtheorem{rem}[thm]{Remark}

\numberwithin{equation}{section}

\newcommand{\spinc}{\mathrm{Spin}^c}

\begin{document}
\title[]{Action of the cork twist on Floer homology}%
\author{ Selman Akbulut, \c{C}a\u{g}ri Karakurt}%
\thanks{The first named author is partially supported by NSF grant DMS 9971440}
\thanks{ The second named author is  supported by a Simons fellowship.}
\address{Department of Mathematics Michigan State University \\
         East Lansing 48824 MI, USA }%
\email{ akbulut@math.msu.edu}
\address{Department of Mathematics
The University of Texas at Austin\\
2515 Speedway, RLM 8.100
Austin, TX 78712 }%
\email{ karakurt@math.utexas.edu}%

\date{\today}
\begin{abstract} 
We utilize the Ozsv\'ath-Szab\'o contact invariant to detect the action of involutions on certain homology spheres that are surgeries on symmetric links, generalizing a previous result of Akbulut and Durusoy. Potentially this may be useful to detect different smooth structures on $4$-manifolds by cork twisting operation.
\end{abstract}


%
\maketitle

\section{Introduction}

Any two different smooth structures of a closed simply connected $4$-manifold are related to each other by a cork twisting operation \cite{M}, and the cork can be assumed to be a Stein manifold \cite{AM} (see \cite{AY} and \cite{AY1} for applications). A quick way to generate corks, which was used in \cite{AY2}, is from  symmetric links as follows: Let $L$ be a link in $S^3$ with two components $K_1\cup K_2$. Suppose that $L$ satisfies the following:

\vspace{0.2cm}

\begin{itemize}
\item[({\bf 1})]  Both components $K_1$ and $K_2$ are unknotted.
\item[({\bf 2})] There is an involution of $S^3$  exchanging $K_1$ and $K_2$.
\item[({\bf 3})]The linking number of $K_1$ and $K_2$ is $\pm 1$ (for some choice of orientations).

\end{itemize}

\vspace{0.23cm}


\noindent From this we can construct a $4$--manifold $W(L)$ by carving out a disk bounded by $K_1$ from $4$--ball, and attaching a $2$--handle along $K_2$ with framing $0$. Therefore a handlebody diagram of $W(L)$ is given by a  planar projection of $L$ decorated by a dot on $K_1$ (cf. \cite{A}), and a $0$ on top of $K_2$. We also require that the $4$--manifold $W(L)$ admits an additional structure:

\vspace{0.2cm}

\begin{itemize}
\item[({\bf 4})]  The handlebody of $W(L)$ described above is induced by a Stein structure. 
\end{itemize}

\vspace{0.2cm}
   
\noindent The last condition can be reformulated as follows:

\vspace{0.2cm}

\begin{itemize}
\item[({\bf 4'})]  Regard $K_2 \subset S^1\times S^2=\partial S^1\times B^3$ equipped with the unique Stein fillable contact structure. Then the  maximal Thurston-Bennequin number of $K_2$ is at least $+1$.    
\end{itemize}
   
 \vspace{0.2cm}
  
\noindent We will call links satisfying conditions $(1)-(4)$ \emph{admissible}. Examples of admissible links are given in  Figure \ref{fig:exam1}. These examples are first studied in \cite{AY}. Note that the Hopf link is not admissible as it does not satisfy condition \textbf{(4)} even though  the corresponding $4$--manifold is $B^4$ which admits a Stein structure.
Condition \textbf{(3)} ensures that $W(L)$ a contractible $4$--manifold. Hence its boundary is a homology sphere. By condition \textbf{(2)}, we have an involution $\tau$ on $\partial W(L)$ that is obtained by exchanging the components of $L$. The significance this involution is indicated by the following theorem:

\begin{figure}[h]
	\includegraphics[width=.40\textwidth]{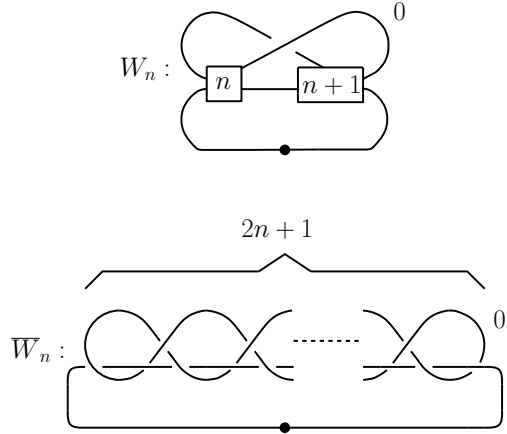}
	\caption{Examples of corks associated to certain symmetric links}
	\label{fig:exam1}
\end{figure}

\begin{theo}\label{main}
Let $L$ be an admissible link. The involution $\tau:\partial W(L)\to \partial W(L)$ acts non trivially on the Heegaard Floer homology group $HF^+(-\partial W(L))$
\end{theo}

\noindent This result generalizes \cite{AD}, and it implies a result from \cite{AY2}; namely $\tau$ does not extend to $W(L)$ as a diffeomorphism, even  though it extends as a homeomorphism. Therefore $W(L)$ is a cork in the sense of \cite{AY}. The involution $\tau$ is called the \emph{cork twist}. Theorem \ref{main} was proved for the Mazur manifold $W_1$ in \cite{A1}, and  \cite{S}(for instanton Floer homology), and  \cite{AD} (for Heegaard Floer homology). Unlike the arguments in these papers, we do not explicitly find the Floer homology of $\partial W(L)$, and calculate the homomorphism the cork twist induces. Instead, we show that the homomorphism permutes two different distinguished Floer homology classes. These classes $c^{+}(\xi)$ are naturally associated to a cork via the induced contact structure $\xi$ on the boundary.  This suggest that a cork should be considered with its contact homology class $(W, c^{+}(\xi))$, to be used as a tool for checking nontriviality of its involution.   In the course of our proof,  we will incorporate several techniques developed in \cite{AY},   \cite{AO} and   \cite{P}.

\vspace{.05in}

Organization of this paper is as follows. We will review some standard facts about Stein manifolds, and Heegaard Floer homology respectively in the next two sections. We shall deduce Theorem \ref{main}  from a slightly stronger result (Theorem \ref{distinct}). This result  along with some easy consequences  are discussed in section \ref{sec:main}.

\section{stein manifolds and their symplectic compactifications }

Our aim in this section is to review the proof of the embedding theorem of Stein manifolds into closed symplectic manifolds in dimension four as given in \cite{AO}. See \cite{LM} for an alternative proof.   We assume that the reader is familiar with the basics of contact geometry, open book decompositions and Lefschetz fibrations (cf. \cite{OzS}). We start by recalling the topological characterization of Stein manifolds.

\vspace{.1in} 

\begin{thm}\label{eli}( \cite{E})
Let $W=B^4\cup(1\text{-handles})\cup (2\text{-handles})$ be a  $4$-dimensional handlebody with one $0$-handle and no $3$ or $4$-handles. Then:
\begin{enumerate}
\item The standard Stein structure on $B^4$ can be extended over the $1$-handles so that the manifold $W_1:=B^4\cup (1\text{-handles})$ becomes Stein.
\item If each $2$-handle is attached to $\partial W_1$ along a Legendrian knot with framing one less than the Thurston-Bennequin framing, then the complex structure on $W_1$ can be extended over the $2$-handles  making $W$ a Stein manifold.
\item The handle decomposition of $W$ is induced by a strictly plurisubharmonic  Morse function.
\end{enumerate} 
\end{thm}

By Theorem \ref{eli} (see also \cite{Go}), we represent Stein manifolds  by special kind of handlebody diagrams which contain handles of index up to two and the attaching circles of the two handles are all Legendrian (i.e. they have horizontal cusps instead of  vertical tangencies and the smaller slope strand is over the bigger slope strand at each crossing). For a Legendrian knot, the Thurston-Bennequin number ($tb$ for short) is defined to be the writhe minus the half of the number of cusps. In these special diagrams we understand that the framing on each $2$--handle is one less than $tb$ as in item $(2)$ in Theorem \ref{eli}. By abuse of language, a handle decomposition as in Theorem \ref{eli} is called a Stein structure.
 
\vspace{.1in}
Similarly for a given contact manifold $(Y,\xi)$, one can attach $1$--handles and $tb-1$ framed $2$--handles to $Y\times \{1 \} \subset Y\times [0,1]$, in order to form a \emph{Stein cobordism} built on $(Y,\xi)$.
\vspace{.1in}

The embedding theorem relies on the fact that Stein manifolds are equivalent to \emph{positive allowable Lefschetz fibrations} ($PALF$ for short). Recall that a $PALF$ on a $4$--manifold $W$ is a Lefschtetz fibration over a disk whose fibers have non-empty boundary and vanishing cycles are all non-separating curves. The restriction of a $PALF$ on the boundary $Y=\partial W$ is an open book decomposition whose monodromy can be written as a product of right handed Dehn twists. By \cite{AO1}, every Stein $4$--manifold admits a $PALF$ (and every $PALF$ has a Stein structure). The construction is algorithmic where input is a handlebody decomposition of a Stein $4$--manifold $W$ as described in Theorem \ref{eli} and the output is a $PALF$ of $W$ which is unique up to positive stabilization. Moreover, it is proved in \cite{P} that the open book induced by this $PALF$ is compatible (in the sense of \cite{G}) with the contact structure $\xi$ induced by the Stein structure.
 
 \vspace{.1in}
 
Given a Stein manifold $W$, fix a compatible $PALF$ as in the previous paragraph. Then one can extend this $PALF$ to a Lefschetz fibration over a closed manifold.   Here is a  sketch of what to do: First, recall the chain relation in the mapping class group of a surface. Let $\Sigma_{g}$ be a surface of genus $g$. Let  $\beta_1,\beta_2,\cdots,\beta_{2g}$ be a set of non-separating simple closed curves such that the following hold.
\begin{itemize}
	\item $|\beta_i \cap \beta_j|=1$ if $|i-j|=1$.
	\item $|\beta_i \cap \beta_j|=0$ if $|i-j|>1$.
\end{itemize}

\vspace{.05in}

\begin{prop}\label{chainrel}Let $t_{\alpha}$ denote the mapping class of the right handed Dehn twist about a curve $\alpha$. Then following relation holds in the mapping class group of $\Sigma_{g}$.

\begin{equation}\label{eq:chain}
\displaystyle (t_{\beta_1}t_{\beta_2}\cdots t_{\beta_{2g}})^{4g+2}=1
\end{equation}
\end{prop}

Now on the Stein manifold $W$, we first choose a $PALF$ which induces an open book on $Y=\partial W$ with connected binding. This open book is compatible with the induced contact structure $\xi$. Then we attach a $2$--handle along the binding with $0$--framing relative to the page framing. Denote the corresponding cobordism by $V_0:Y\to Y_0$. By  \cite{El2}, $V_0$ can be equipped with a symplectic structure extending the one defined on a collar neighborhood of $Y$ . On the other hand, $Y_0$ is a surface bundle over circle whose monodromy can be written as a product of right handed Dehn twists along non-separating curves. Let $F$ denote a generic fiber in $Y_0$. Next, we use the chain relation of Proposition \ref{chainrel} to trivialize the monodromy by attaching $-1$ framed $2$--handles: Write the monodromy of $Y_0$ as a product of right handed Dehn twists $t_{\gamma _1}\cdots t_{\gamma _n}$, where each $\gamma_i$ is a non-separating curve on $F$. There is a diffeomorphism of $F$ identifying $\gamma_i$ with $\beta_1$ for each $i=1,\cdots,n$. Using this diffeomorphism and the chain relation, we can write $t_{\gamma_i}^{-1}$ as a product of right handed Dehn twists.  Attaching $2$--handles as necessary, we can trivialize the monodromy.   Finally we attach a copy $F\times D^2$ to get a cobordism $V_1:Y_0\to \emptyset $. Note that $V_1$ itself admits a Lefschetz fibration (with closed fibers) over disk.  The closed $4$--manifold $X:=W\cup V_0 \cup V_1$ naturaly admits a Lefshcetz fibration over $S^2$. It is also possible to show that the Lefschetz fibration has a section, and hence symplectic, and that the construction can be made to guarantee that $b_2^+(X)\geq 2$. In other words, $V=V_0\cup V_1$ is a \emph{concave symplectic filling} for $(Y,\xi)$. We summarize this construction in the following statement.

\begin{theo}(\label{concavefill}\cite{AO1})
Every Stein fillable contact manifold $(Y,\xi)$ admits a concave symplectic filling $V=V_0\cup V_1$, where $V_0$ is the cobordism $Y\to Y_0$ corresponding to a $2$--handle attachment along the binding of an open book compatible with $\xi$, and $V_1$ admits a Lefschetz fibration over disk with closed fibers, which extends the  fibration on $Y_0$. Moreover, $V$ can be chosen in such a way that $b_2^+(V)\geq 2$.
\end{theo}


\section{heegaard floer homology}

Heegaard Floer homology  (\cite{OS4}, \cite{OS5}) is a type of Lagrangian Floer homology for the symmetric product of a Heegaard surface of a $3$--manifold. There are several versions denoted by $\widehat{HF}(Y,\mathfrak{t})$, $HF^+ (Y,\mathfrak{t}), HF^- (Y,\mathfrak{t}),HF^\infty (Y,\mathfrak{t})$. All of these groups are invariants of a  $3$-manifold $Y$ with a $\mathrm{Spin}^c$ structure $\mathfrak{t}$. When $c_1(\mathfrak{t})$ is torsion these groups are $\mathbb{Q}$-graded. Each one admits an endomorphism $U$ of degree $-2$ which makes all of them $\mathbb{Z}[U]$ modules (the action of $U$ on $\widehat{HF}$ is trivial). They also satisfy the property that any $\mathrm{Spin}^c$ cobordism 
 $(M,\mathfrak{s}):(Y_1,\mathfrak{t}_1)\to (Y_2,\mathfrak{t}_2)$, 
 from $(Y_1,\mathfrak{t}_1)$ to  $(Y_2,\mathfrak{t}_2)$, induces a homomorphism 
\begin{equation}\label{cobhom}
F^\circ _{(M,\mathfrak{s})}:HF^\circ(Y_1,\mathfrak{t}_1)\to HF^\circ(Y_2,\mathfrak{t}_2)
\end{equation}
where $HF^\circ$ represents any of $\widehat{HF}, HF^+, HF^-$,or $HF^\infty$. When both $c_1(\mathfrak{t}_1)$ and $c_1(\mathfrak{t}_2)$ are torsion, these homomorphisms shift degree by

\begin{equation}\label{eqn:degree}
d(M,\mathfrak{s})=\frac{c_1(\mathfrak{s})^2-3\sigma (M)-2\chi (M)}{4}
\end{equation}

These homomorphisms also satisfy the following \emph{composition law}: If $(M_1,\mathfrak{s}_1):(Y_1,\mathfrak{t}_1)\to (Y_2,\mathfrak{t}_2)$ and $(M_2,\mathfrak{s}_2):(Y_2,\mathfrak{t}_2)\to (Y_3,\mathfrak{t}_3)$ are two $\spinc$ cobordisms, and $F^\circ_{M_1,\mathfrak{s}_1}$ and $F^\circ_{M_2,\mathfrak{s}_2}$ are the induced homomorphism their composition is given by

\begin{equation}\label{eqn:composition}
F^\circ_{M_2,\mathfrak{s}_2}\circ F^\circ_{M_1,\mathfrak{s}_1}=\sum_{\mathfrak{s}\in \spinc (M_1\cup W_2):\mathfrak{s}|_{M_i}=\mathfrak{s}_i}F^\circ _{M_1\cup M_2,\mathfrak{s}}
\end{equation}

The following long exact sequence exists for every $\mathrm{Spin}^c$ $3$--manifold $(Y,\mathfrak{t})$ and it is natural under cobordism--induced homomorphisms.

\begin{equation}\label{longexct}
\begin{CD}
\cdots @>\delta>> HF^-(Y,\mathfrak{t})@>\iota>> HF^\infty(Y,\mathfrak{t})  @>\pi>>HF^+(Y,\mathfrak{t})@>\delta>>\cdots
\end{CD}
\end{equation}

\noindent The connecting homomorphism  $\delta$ is  an isomorphim between $\mathrm{coker}(\pi)$ and $\ker{\iota}$. We denote these groups respectively by $HF^+_\mathrm{red}(Y,\mathfrak{t})$ and $HF^-_\mathrm{red}(Y,\mathfrak{t})$, and call both of them by the same name: the \emph{reduced} Heegaard Floer homology. 

Heegaard Floer homology also provides a $4$-manifold invariant, \cite{OS6}. To review the definition of this invariant first recall the mixed homomorphisms. Let $(M,\mathfrak{s}):(Y_1,\mathfrak{t}_1)\to (Y_2,\mathfrak{t}_2)$ be a $\mathrm{Spin}^c$ cobordism with $b_2^+(M)\geq 2$. Every such cobordism admits an admissible cut; That is to say  $M$ can be decomposed as a union of two codimension zero sub-manifolds $M_1$ and $M_2$ with  $b_2^+(M_i)\geq 1$, $i=1,2$, and $\delta H^1(N)\subseteq H^2(M)$ is trivial where $N$ is the common boundary of these two sub-manifolds. Let $\mathfrak{s}_i=\mathfrak{s}|_{M_i}$ and $\mathfrak{t}=\mathfrak{s}|_{N}$. Then we have the following commutative diagram. 

$$
\begin{CD}
HF^+(Y_1,\mathfrak{t_1})@>\delta>> HF^-(Y_1,\mathfrak{t_1})  @>\iota>>HF^\infty(Y_1,\mathfrak{t_1})\\
@VV F^+_{M_1,\mathfrak{s}_1} V @VV F^-_{M_1,\mathfrak{s}_1} V @VV F^\infty_{M_1,\mathfrak{s}_1}=0 V\\
HF^+(N,\mathfrak{t})@>\delta>> HF^-(N,\mathfrak{t})  @>\iota>>HF^\infty(N,\mathfrak{t})\\
@VV F^+_{M_2,\mathfrak{s}_2} V @VV F^-_{M_2,\mathfrak{s}_2} V @VV F^\infty_{M_2,\mathfrak{s}_2}=0 V\\
HF^+(Y_2,\mathfrak{t_1})@>\delta>> HF^-(Y_2,\mathfrak{t_1})  @>\iota>>HF^\infty(Y_2,\mathfrak{t_1})
\end{CD}
$$

\vspace{.05in}

The mixed homomorphism:  
$$F^{\mathrm{mix}}_{(M,\mathfrak{s})}: HF^-(Y_1,\mathfrak{t}_1)\to HF^+(Y_2,\mathfrak{t}_2)$$

\noindent is defined by the composition $F^+_{(M,\mathfrak{s}_2)}\circ \delta ^{-1} \circ F^-_{(M,\mathfrak{s}_1)}$. In \cite{OS6}, It is proved that the mixed homomorphism is independent of the admissible cut.

\vspace{.05in}
To define the $4$--manifold invariant, we need to recall  the Heegaard-Floer homology groups of  the $3$--sphere $S^3$ with its unique $\mathrm{Spin}^c$ structure:
\begin{displaymath}
   HF^+_n(S^3) = \left\{
     \begin{array}{ll}
       \mathbb{Z} & \mathrm{If} \; \mathrm{n\, is\, even\, and} \; n\geq{0}\\
       0 & \mathrm{If} \; \mathrm{n \, is\, odd,} 
     \end{array}
   \right.
\end{displaymath} 

\begin{displaymath}
   HF^-_n(S^3) = \left\{
     \begin{array}{ll}
       \mathbb{Z} & \mathrm{If} \; \mathrm{n\, is\, even\, and} \; n\leq{-2}\\
       0 & \mathrm{If} \; \mathrm{n \, is\, odd,} 
     \end{array}
   \right.
\end{displaymath} 

Now, let $X$ be a closed $4$--manifold with $b_2^+(X)\geq 2$ and $\mathfrak{s}$ be a $\mathrm{Spin}^c$ structure on $X$. For simplicity, assume that $b_1(X)=0$. Puncture $X$ at two points and regard it as a cobordism from the $3$--sphere to itself. Let $\Theta ^\pm_{(n)}$  denote the generator of $HF_n ^\pm(S^3)$. The \emph{Ozsv\'ath-Szab\'o} $4$--manifold invariant is a linear map $\Phi_{X,\mathfrak{s}}:\mathbb{Z}[U]\to\mathbb{Z}$ which is described as follows:  $\Phi_{X,\mathfrak{s}}(U^n)$   can be characterized uniquely by the  formula. 
$$F^{\mathrm{mix}}_{X,\mathfrak{s}}(U^n\Theta^-_{(-2)})=(\Phi_{X,\mathfrak{s}}(U^n))\Theta^+_{(0)}$$

The $4$- manifold invariant is zero on elements of degree not equal to $d(X,\mathfrak{s})$. A $\mathrm{Spin}^c$ structure $\mathfrak{s}$ on $X$ is called a \emph{basic class} if $\Phi_{X,\mathfrak{s}}\not \equiv 0$. Finding the set of all basic classes of a given $4$--manifold is an important problem in low-dimensional topology. The \emph{adjunction inequality} gives a very strong restriction on the set of basic classes.

\begin{theo}\label{adjuntion}(\cite{OS6})
Let $X$ be a closed $4$--manifold. Let $\Sigma \subset X$ be a homologically non-trivial embedded surface with genus $g\geq 1$ and with non-negative self-intersection number. Then for each $\mathrm{spin}^c$ structure $\mathfrak{s}\in\mathrm{Spin}^c(X)$ for which $\Phi_{X,\mathfrak{s}} \neq 0$, we have that
\begin{equation}
\left | \langle c_1(\mathfrak{s}),[\Sigma]\rangle \right |+[\Sigma]\cdot[\Sigma] \leq 2g-2
\end{equation}
\end{theo}

The following is another version of the adjunction inequality along with a non-vanishing result of the $4$--manifold invariant for Lefschetz fibrations.

\begin{theo}\label{OSadjunc}(\cite{OS7}) Let $\pi :X\to S^2$ be a relatively minimal Lefschetz fibration over sphere with generic fiber $F$ of genus $g>1$, and $b_2^+>1$. Then for the canonical $\mathrm{spin}^c$ structure $\mathfrak{s}$ the map $F_{X,\mathfrak{s}}^{\mathrm{mix}}$ sends the generator of $HF^{-}_{-2}(S^3)$ to the generator of $HF_0 ^+(S^3)$ (and vanishes on the rest of $HF^{-}(S^3)$). In particular $\mathfrak{s}$ is a basic class. For any other $\mathrm{spin}^c$ structure $\mathfrak{t}\neq \mathfrak{s}$ with $\langle c_1(\mathfrak{t}),[F]\rangle \leq 2-2g=\langle c_1(\mathfrak{s}),[F]\rangle$, the map $F_{X,\mathfrak{t}} ^{\mathrm{mix}}$ vanishes. 

\end{theo}

Given a contact structure $\xi$ on $Y$, let $\mathfrak{t}_\xi$ be the induced $\mathrm{Spin}^c$ structure.  A Heegaard-Floer (co-)homology class $c^+\in HF^+(-Y,\mathfrak{t}_\xi)$, which is an invariant of the isotopy class of $\xi$, is constructed in \cite{OS2} as follows: Take an adapted open book decomposition for $\xi$ which has a unique binding component. One can always find such an open book by doing positive stabilizations to any adapted open book as necessary. Let $Y_0$ denote the result of the $0$--surgery on the binding, and $V_0:Y\to Y_0$ be the associated cobordism. Naturally, $Y_0$  admits a fibration over circle. Let $\mathfrak{t}_0$ the $\mathrm{Spin}^c$ structure corresponding tangent plane distribution of fibers. 

\begin{prop}(\cite{OS7})
$HF^+(-Y_0,\mathfrak{t}_0)=\mathbb{Z}$
\end{prop}

Let $c$ be a generator of $HF^+(-Y_0,\mathfrak{t}_0)$. It can be shown  that there is a unique extention $\mathfrak{s}$ of the $\mathrm{Spin}^c$ structure $\mathfrak{t}_0$ over the cobordism $W$. The contact invariant is defined to be the image of $c$ under the homomorphism which induced by the $\mathrm{Spin}^c$ cobordism $(V_0,\mathfrak{s})$.

\begin{defn}
$c^+(\xi):=F^+_{V_0,\mathfrak{s}}(c)\in HF^+(-Y,\mathfrak{t}_\xi)/(\pm 1)$.
\end{defn}

Note that we turned $V_0$ upside down in this construction. This invariant is independent of the choice of the adapted open book decomposition used in its definition. 

\vspace{.05in}

By using this definition along with the adjunction inequality and the symplectic compactification theorem, it can be proven that  the contact invariant of a Stein fillable contact structure is in the image of the mixed homomorphism of some concave filling, \cite{P}. From Theorem \ref{concavefill}, we know that every Stein fillable contact manifold $(Y,\xi)$ admits a concave symplectic filling $V=V_0\cup V_1$ where  $V_0$ is a cap off cobordism and $V_1$ is a Lefschetz fibration over a disk. Let $\mathfrak{s}$ be the canonical $\spinc$ structure.

\vspace{.05in}

\begin{lemm}(\cite{P})\label{concavehit}
$F^{\mathrm{mix}}_{V,\mathfrak{s}}(\Theta^{-}_{(-2)})=\pm c^+(\xi)$ if $c_1(\xi)$ is torsion.
\end{lemm}

This lemma will play a key role in our argument. Its proof relies on the special topology of the concave filling constructed in Theorem \ref{concavefill}. One would hope to prove it for arbitrary concave fillings (with $b_2^+\geq 2$), but the authors do not know  how to do it in full generality.

We are going to need a variant of  Lemma \ref{concavehit} where one is allowed to add any Stein cobordism to the concave filling.

\begin{lemm}\label{concavehitit}
Let $(Y,\xi)$ be a Stein fillable contact manifold with torsion $c_1(\xi)$. Let $M$ be any Stein cobordism built on $(Y,\xi)$ which does not contain any $1$--handles. Then $M$ can be extended to a concave filling $V$ of $(Y,\xi)$ such that  $F^{\mathrm{mix}}_{V,\mathfrak{s}}(\Theta^{-}_{(-2)})=\pm c^+(\xi)$.
\end{lemm}

\begin{proof}
Let $(Y_1,\xi_1)$ be the convex end of $M$. Take a Stein filling of $(Y,\xi)$ and glue it to $M$ in order to obtain a Stein filling of $(Y_1,\xi_1)$. Pick a PALF of this Stein manifold and apply the algorithm in the proof of the Theorem \ref{concavefill} to find a concave filling $V=V_0\cup V_1$ of $(Y_1,\xi_1)$. If we glue this to $M$, we get a concave filling $V'$ of $(Y,\xi)$. By changing the order of some $2$--handle attachments we can write this as  $V'=V_0\cup V'_1$, where $V'_1=M\cup V_1$ which is a Lefschetz fibration on disk. Now apply Lemma \ref{concavehit}.
\end{proof}

\vspace{0.2cm}

It is possible to define a relative version of Ozsv\'ath--Szab\'o $4$--manifold invariant in the presence of a contact structure on the boundary. For let $W$ be a $4$--manifold with connected boundary and $\xi$ be a contact structure on $\partial W$. Given a $\spinc$ structure $\mathfrak{s}$ on $W$, the relative invariant $\Phi_{W,\mathfrak{s}}(\xi)\in \mathbb{Z}/\pm 1$ is uniquely characterized by the formula

$$F^+_{W,\mathfrak{s}}(c^+(\xi))=\Phi_{W,\mathfrak{s}}(\xi)\Theta^+_{(0)}.$$

Again we puncture $W$ and turn it upside down to regard it as a cobordism from $-\partial W$ to $S^3$. A twisted version of this invariant is conjecturally equivalent to the relative Seiberg--Witten invariant defined in \cite{KM}.

\section{main theorem}\label{sec:main}

With all the necessary tools in hand, we are ready to prove our main result. Henceforth suppose that $W$ is a cork corresponding to an admissible link $L$. Let $\xi$ be the induced contact structure for some choice of a Stein structure on $W$. Let $\tau$ be the involution on $\partial W$ obtained by exchanging the components of $L$. Let $\xi '$ be the pull back contact structure $\tau^*\xi$.

\vspace{.05in} 

\begin{theo}\label{distinct}
The contact invariants $c^+(\xi)$ and $c^+(\xi ')$ are distinct elements of $HF^+(-\partial W)$. Moreover, both  elements descent non-trivially to $HF^+_{\mathrm{red}}(-\partial W)$. 
\end{theo}

\vspace{.05in} 

\begin{proof}
The trick is to inflate the cork using a Stein handle so that the cork twist changes the framing on the handle. This trick was first used in \cite{AY1} to generate exotic Stein manifold pairs. We start by attaching a $2$-handle to $\partial W$ along a trefoil with framing 1 as in the left hand side of Figure \ref{fig:twist}. Thanks to the non-trivial linking with the $1$-handle, this handle attachment induces a Stein cobordism $M$ built on $(\partial W, \xi)$. On the right hand side of the same figure, however,  the handle attachment can not be realized as a Stein cobordism, because the maximum Thurston-Bennequin number of trefoil is $1$.

\begin{figure}[h]
	\includegraphics[width=.55\textwidth]{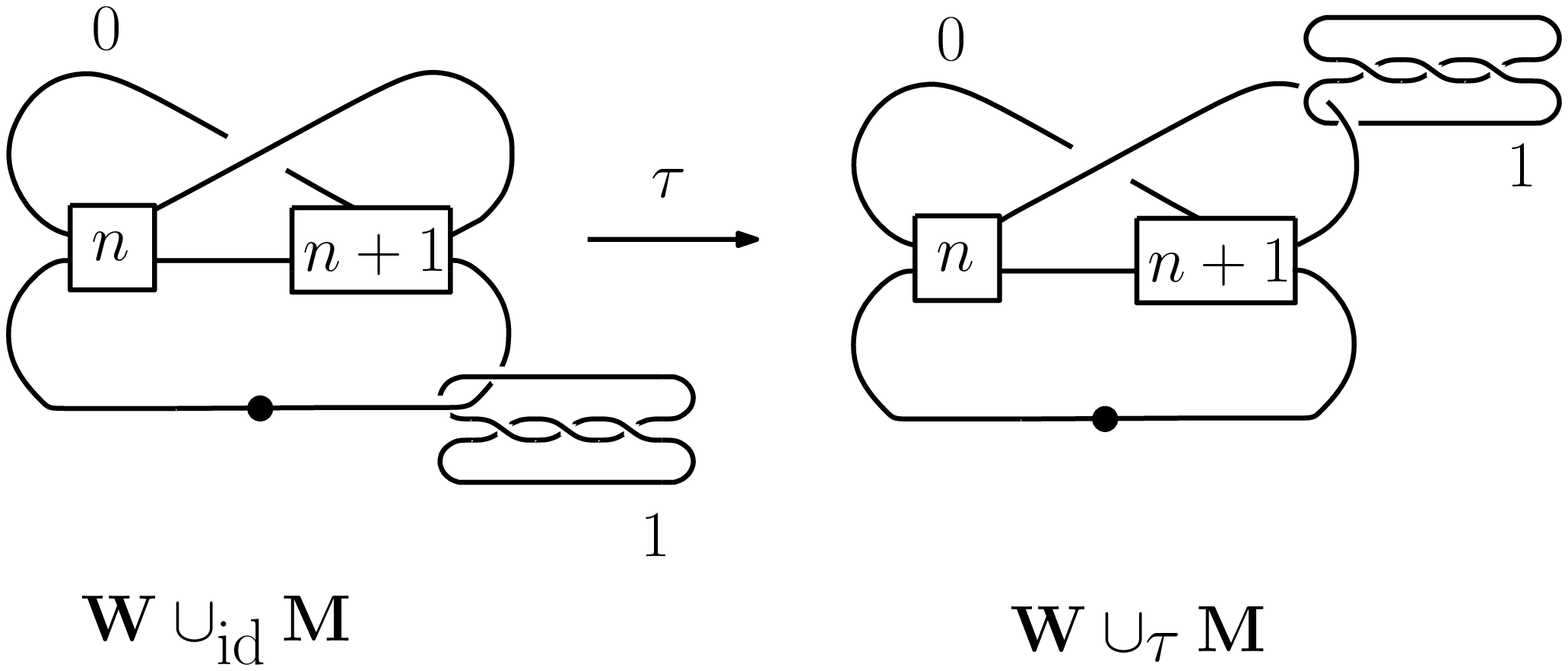}
	\caption{}
	\label{fig:twist}
\end{figure}

\vspace{.05in} 

Next we apply Lemma \ref{concavehitit}, in order  to extend $M$ to a concave filling $V$ of $(\partial W, \xi)$ whose mixed homomorphism hits the contact invariant $c^+(\xi)$.

\vspace{.05in}

The symplectic manifold $X:=M\cup V$ admits a relatively minimal Lefschetz fibration. Let $\mathfrak{s}$ denote its canonical $\spinc$ structure.  By Theorem \ref{OSadjunc}, Lemma \ref{concavehitit} and the compostion law, we have
\begin{eqnarray}\nonumber
\Theta^+_{(0)}&=&F^{\mathrm{mix}}_{W\cup V, \mathfrak{s}}(\Theta^-_{(-2)})\\
\nonumber &=& F^+_{W,\mathfrak{s}}\circ F^{\mathrm{mix}}_{V,\mathfrak{s}}(\Theta^-_{(-2)})\\
\nonumber &=&\pm F^+_{W,\mathfrak{s}}(c^+(\xi)).
\end{eqnarray}

\noindent In particular $F^+_{W,\mathfrak{s}}(c^+(\xi))\neq 0$. On the other hand if we remove the cork $W$ from $X$ and reglue it using the cork twist $\tau$, we obtain the manifold $X':=W\cup_\tau V$  which is homeomorphic to $X$. The right hand side of Figure \ref{fig:twist} is $W\cup _\tau M$ and it embeds into $X'$. In that figure the trefoil represents an embedded torus of self intersection $1$, which violates the adjunction inequality (Theorem \ref{adjuntion}). Therefore $X'$ has no basic class. This implies

\begin{eqnarray}\nonumber
0&=&F^{\mathrm{mix}}_{W\cup_\tau V, \mathfrak{s}}(\Theta^-_{(-2)})\\
\nonumber &=& F^+_{W,\mathfrak{s}}\circ \tau^*\circ F^{\mathrm{mix}}_{V,\mathfrak{s}}(\Theta^-_{(-2)})\\
\nonumber &=&\pm F^+_{W,\mathfrak{s}}(\tau^*(c^+(\xi))).
\end{eqnarray}

This proves that $c^+(\xi)$ and $c^+(\xi ')=\tau^*(c^+(\xi))$ are distinct. To prove the last statement, note that up to sign the $U$--equivariant involution $\tau^*:HF^+(-\partial W)\to HF^+(-\partial W)$ fixes the image of $HF^\infty(-\partial W)$ under the homomorphism $\pi$ in Equation \ref{longexct}. Since $c^+(\xi)$ and $(c^+(\xi '))$ are not fixed by $\tau$ they descent non-trivially to the $\mathrm{coker} (\pi )$.
\end{proof}

\noindent The following corollary was previously proved in \cite{AM} by using different techniques.

\begin{cor}
The contact structure $\xi$ and $\xi '$ are homotopic, contactomorphic but not isotopic.
\end{cor}

\begin{proof}
We use the homotopy classification of $2$--plane fields on a $3$--manifold (see \cite{Go}). Since $\partial W$ is an integral homology sphere, $\xi$ and $\xi '$ have the same two dimensional invariant. These two have also the same $3$--dimensional invariant because they can be connected by a Stein cobordism which is topologically trivial: Simply take the symplectizations of $(\partial X ,\xi)$ and $(\partial X ,\xi)$, and glue two ends using $\tau$. This proves that $\xi$ and $\xi '$ are homotopic. By definition $\tau$ defines a contactomorphism between these two contact structures. Theorem \ref{distinct} shows that they are not isotopic. \end{proof}

\noindent Next corollary was first proved in \cite{A2} for Mazur manifold, and in \cite{AY2} in the general form. 

\begin{cor}
The cork twist $\tau:\partial W\to \partial W$ does not extend inside of $W$ as a diffeomorphism.
\end{cor}

\begin{proof} Let $\mathfrak{s}$ be the unique $\spinc$ structure on $W$. Let $\xi$ be the contact structure induced by a Stein structure.  The proof of Theorem \ref{distinct} shows that $F_{W,\mathfrak{s}}(c^+(\xi))=\Theta^+_{(0)}$ and $F_{W,\mathfrak{s}}(\tau^* c^+(\xi))=0$. This shows that the relative invariants $\Phi_{W,\mathfrak{s}}(\xi)$ and $\Phi_{W,\mathfrak{s}}(\tau^* \xi)$ are different.
\end{proof}

The following corollary was first proved in \cite{AM1} by using adjunction inequality. Let us first introduce a terminology. Two $\spinc$ manifolds $(X,\mathfrak{s})$ and $(X',\mathfrak{s}')$ are said to be fake copies of each other if they are homeomorphic but not diffeomorphic.

\begin{cor}\label{cor:fake}
The cork $W$ can be symplectically embedded in some closed symplectic $4$--manifold $X$ so that removing $W$ and regluing it via cork twist produces a fake copy $X$, with its canonical $\spinc$ structure.

\end{cor}

\begin{proof}
Pick a concave symplectic filling $V$ of $(\partial W, \xi)$ as in Theorem \ref{concavefill}. The manifold $X=W\cup V$ is simply connected and symplectic. Let $\mathfrak{s}$ be the canonical $\spinc$ structure of $X$. Now, remove the cork and reglue it using cork twist. The manifold $X'=W\cup_\tau V$ is simply connected and has the same intersection form as $X$.  By Freedman's theorem there is a homeomorphism $f:X\to X'$. Let $\mathfrak{s}'=f_*(\mathfrak{s})$. We will prove that $\mathfrak{s}'$ is not a basic class. By Lemma \ref{concavehit}

\begin{eqnarray}\nonumber
F^\mathrm{mix}_{X',\mathfrak{s}'}(\Theta^-_{(-2)})&=&F^+_{W,\mathfrak{s}}\circ \tau^* \circ F^{\mathrm{mix}}_{V,\mathfrak{s}}(\Theta^-_{(-2)})\\
\nonumber &=& \pm F^+_{W,\mathfrak{s}}(\tau^*(c^+(\xi)))\\
\nonumber &=& 0.
\end{eqnarray}
\end{proof}

\begin{rem}
Note that if the inflated cork $W\cup_{id}M$ (of Theorem 1.1) embeds symplectically into a symplectic manifold $X$, then the cork twist produces a fake copy of $X$ since in the cork twisted manifold the adjunction inequality fails. This is what is used in \cite{AM1} and \cite{AY2}.
\end{rem}

\begin{rem}
Let us formulate Corollary \ref{cor:fake} in other terms: There is a concave filling $V$ of $(\partial W, \xi)$ such that attaching $V$ to $W$  in two different ways produce two closed $\spinc$ $4$--manifolds $(X,\mathfrak{s})$ and $(X',\mathfrak{s}')$ that are fake copies of each other.  A large family of concave fillings satisfies this condition and presumably this also holds for all concave fillings (with $b_2^+(V)\geq 2$). Proving latter, however, requires a generalization of Lemma \ref{concavehit} for arbitrary concave fillings. Once this is done, one would be able to show the following:
\end{rem}

\noindent \textbf{Conjecture:} If a cork embeds symplectically into any symplectic manifold $X$, such that  $b_2^+(X)\geq 2$, then the cork twist produces a fake copy of $X$.

\end{document}